\documentclass[fleqn, 11pt]{article}

\usepackage{amsmath,amssymb,amsthm,pifont}

\title{Rosser provability and normal modal logics}
\author{Taishi Kurahashi}
\date{}

\theoremstyle{plain}
\newtheorem{thm}{Theorem}[section]
\newtheorem{lem}[thm]{Lemma}
\newtheorem{prop}[thm]{Proposition}
\newtheorem{cor}[thm]{Corollary}

\newtheorem{prob}[thm]{Problem}

\theoremstyle{definition}
\newtheorem{defn}[thm]{Definition}

\newtheorem{rem}[thm]{Remark}

\newcommand{\PA}{{\sf PA}}
\newcommand{\PL}{{\sf PL}}

\newcommand{\PR}{{\sf Pr}^R_T}
\newcommand{\PRg}{{\sf Pr}^R_g}
\newcommand{\PRgp}{{\sf Pr}^R_{g'}}
\newcommand{\Prf}{{\sf Prf}}
\newcommand{\Proof}{{\sf Proof}_T}
\newcommand{\Prov}{{\sf Prov}_T}
\newcommand{\Con}{{\sf Con}}
\newcommand{\K}{{\sf K}}
\newcommand{\KR}{{\sf KR}}
\newcommand{\GL}{{\sf GL}}
\newcommand{\KD}{{\sf KD}}
\newcommand{\KT}{{\sf T}}
\newcommand{\Four}{{\sf KD4}}
\newcommand{\Five}{{\sf KD5}}
\newcommand{\KDR}{{\sf KDR}}
\newcommand{\N}{\mathbb{N}}
\newcommand{\gdl}[1]{\ulcorner#1\urcorner}

\begin{document}

\maketitle

\abstract{
In this paper, we investigate Rosser provability predicates whose provability logics are normal modal logics. 
First, we prove that there exists a Rosser provability predicate whose provability logic is exactly the normal modal logic $\KD$.
Secondly, we introduce a new normal modal logic $\KDR$ which is a proper extension of $\KD$, and prove that there exists a Rosser provability predicate whose provability logic includes $\KDR$. 
}

\section{Introduction}\label{Sec:Intro}

In the paper \cite{Kur}, we raised the problem of the existence of a $\Sigma_2$ representation of each theory $T$ such that the provability logic of the conventional provability predicate constructed from this representation is exactly the modal logic $\KD = \K + \neg \Box \bot$. 
This problem has not been settled yet. 
Here we consider the following more general question: 
Is there a provability predicate whose provability logic is exactly $\KD$? 
In this paper, we give an affirmative answer to this problem by considering Rosser provability predicates. 

Let $T$ be any consistent recursively enumerable extension of Peano Arithmetic $\PA$. 
We say a formula ${\sf Pr}_T(x)$ is a {\it provability predicate of $T$} if it weakly represents the set of all theorems of $T$ in $\PA$, that is, for any natural number $n$, $\PA \vdash {\sf Pr}_T(\overline{n})$ if and only if $n$ is the G\"odel number of some theorem of $T$. 
An {\it arithmetical interpretation based on ${\sf Pr}_T(x)$} is a mapping $f$ from modal formulas to sentences of arithmetic such that $f$ commutes with every propositional connective and $f$ maps $\Box$ to ${\sf Pr}_T(x)$. 
Let $\PL({\sf Pr}_T)$ be the set of all modal formulas $A$ such that $T \vdash f(A)$ for each arithmetical interpretation $f$ based on ${\sf Pr}_T(x)$. 
This set is called the {\it provability logic of ${\sf Pr}_T(x)$}. 
Solovay \cite{Sol76} proved that for each standard $\Sigma_1$ provability predicate ${\sf Pr}_T(x)$ of $T$, if $T$ is $\Sigma_1$-sound, then the provability logic of ${\sf Pr}_T(x)$ is equal to the modal logic ${\sf GL}$. 
This is Solovay's arithmetical completeness theorem. 

On the other hand, Feferman \cite{Fef60} found a $\Pi_1$ representation of a theory $T$ such that $\PA$ proves the consistency statement for the provability predicate ${\sf Pr}_T^F(x)$ constructed from this representation. 
The provability logic $\PL({\sf Pr}_T^F)$ of Feferman's predicate includes the modal logic $\KD$, and it is completely different from ${\sf GL}$. 
The problem of exact axiomatization of $\PL({\sf Pr}_T^F)$ was studied by Montagna \cite{Mon78} and Visser \cite{Vis89}, but it has not been settled yet. 
Shavrukov \cite{Sha94} found a Feferman-like $\Sigma_2$ provability predicate whose provability logic is exactly the modal logic $\KD + \Box p \to \Box ((\Box q \to q) \lor \Box p)$. 

Rosser provability predicate $\PR(x)$ was essentially introduced by Rosser \cite{Ros36} to improve G\"odel's first incompleteness theorem. 
It is well-known that the consistency statement for $\PR(x)$ is provable in $\PA$. 
Then by the proof of G\"odel's second incompleteness theorem, at least one of the principles $({\bf K})$: $\Box(p \to q) \to (\Box p \to \Box q)$ and $({\bf 4})$: $\Box p \to \Box \Box p$ is invalid for each Rosser provability predicate (Whether the principle $({\bf K})$ is valid for Rosser provability predicates was asked by Kreisel and Takeuti \cite{KT74}). 
Actually, Guaspari and Solovay \cite{GS79} and Arai \cite{Ara90} showed that whether $({\bf K})$ or $({\bf 4})$ is invalid for $\PR(x)$ depends on the choice of $\PR(x)$. 
More precisely, by using the modal logical result of Guaspari and Solovay, it can be shown that there exists a Rosser provability predicate for which neither of these principles is valid. 
Also Arai proved the existence of a Rosser provability predicate satisfying $({\bf K})$ and a Rosser provability predicate satisfying $({\bf 4})$. 
%Kurahashi and Kikuchi \cite{KK17} proved the existence of a Rosser provability predicate for which $({\bf K})$ and an additional condition are valid. 

Modal logical investigations of Rosser provability predicates were initiated by Guaspari and Solovay, and continued by Visser \cite{Vis89}, Shavrukov \cite{Sha91} and others. 
In particular, Shavrukov introduced the bimodal logic ${\sf GR}$ for usual provability and Rosser provability, and proved the arithmetical completeness theorem for ${\sf GR}$. 
Although Shavrukov's arithmetically complete logic ${\sf GR}$ does not contain $({\bf K})$ for the modality of Rosser provability as an axiom, it is worth considering $({\bf K})$ for Rosser provability from modal logical viewpoint. 
That is, it is easy to show that the provability logic $\PL(\PR)$ is a normal modal logic if and only if $({\bf K})$ is valid for $\PR(x)$. 
%Thus whether $\PL(\PR)$ is normal or not is dependent on the choice of $\PR(x)$. 
If $\PL(\PR)$ is normal, then $\PL(\PR)$ includes $\KD$. 

In this paper, we investigate Rosser provability predicates whose provability logics are normal. 
In Section \ref{Sec:ACT}, we give an affirmative answer to the problem raised in the first paragraph of this section, that is, we prove that there exists a Rosser provability predicate $\PR(x)$ of $T$ such that $\PL(\PR)$ is exactly $\KD$. 
In Section \ref{Sec:KDR}, we introduce and study a new normal modal logic $\KDR = \KD + \Box \neg p \to \Box \neg \Box p$. 
In particular, we prove that there exists a Rosser provability predicate $\PR(x)$ of $T$ such that $\KDR \subseteq \PL(\PR)$. 
Thus we obtain a Rosser provability predicate whose provability logic is a proper extension of $\KD$. 
Whether there exists a Rosser provability predicate whose provability logic is exactly $\KDR$ is still open.

\section{Preliminaries}\label{Sec:Pre}

The axioms of the modal logic $\K$ are all propositional tautologies in the language of propositional modal logic and the formula $\Box (p \to q) \to (\Box p \to \Box q)$. 
The inference rules for $\K$ are modus ponens, necessitation and substitution. 
Each modal logic $L$ is identified with the set of all theorems of $L$. 
We say a modal logic $L$ is {\it normal} if $\K \subseteq L$ and $L$ is closed under modus ponens, necessitation and substitution. 
For any modal logic $L$ and modal formula $A$, let $L + A$ denote the least normal modal logic whose axioms are those of $L$ and the formula $A$. 
Several normal modal logics are obtained by adding axioms to $\K$ as follows (see \cite{Boo93,CZ97} for more details): 

\begin{enumerate}
	\item $\KD = \K + \neg \Box \bot$. 
	\item $\KT = \K + \Box p \to p$. 
	\item $\Four = \KD + \Box p \to \Box \Box p$. 
	\item $\Five = \KD + \neg \Box p \to \Box \neg \Box p$. 
	\item $\GL = \K + \Box (\Box p \to p) \to \Box p$. 
\end{enumerate}

A {\it Kripke frame} is a tuple $(W, \prec)$ where $W$ is a nonempty set and $\prec$ is a binary relation on $W$. 
A {\it Kripke model} is a tuple $M= (W, \prec, \Vdash)$ where $(W, \prec)$ is a Kripke frame, 
and $\Vdash$ is a binary relation between $W$ and the set of all modal formulas satisfying the usual propositional conditions for satisfaction and the following condition: 
$x \Vdash \Box A$ if and only if for all $y \in W$, $y \Vdash A$ if $x \prec y$.
We say a modal formula $A$ is {\it valid} in a Kripke model $M = (W, \prec, \Vdash)$ if for all $w \in W$, $w \Vdash A$. 
We say that $M$ is finite if $W$ is finite. 
Also we say that $M$ is {\it serial} if for any $x \in W$, there exists $y \in W$ such that $x \prec y$. 

It is known that the modal logic $\KD$ is sound and complete with respect to the class of all finite serial Kripke models. 
Moreover, the following theorem holds. 

\begin{thm}[Kripke completeness theorem for $\KD$ (see \cite{Rau83})]\label{KCKD}
For each modal formula $A$ which is not provable in $\KD$, we can primitive recursively find a finite serial Kripke model in which $A$ is not valid. 
\end{thm}

%We assume that the logical symbols of first-order logic are $\neg, \land$ and $\exists$, and the other symbols such as $\lor$, $\to$ and $\forall$ are introduced as abbreviations. 
Throughout this paper, we assume that $T$ always denotes a recursively enumerable consistent extension of Peano Arithmetic $\PA$ in the language $\mathcal{L}_A$ of first-order arithmetic. 
%We also assume that $\mathcal{L}_A$ contains function symbols for all primitive recursive functions. 
Let $\omega$ be the set of all natural numbers. 
For each $n \in \omega$, the numeral for $n$ is denoted by $\overline{n}$. 
We fix a natural G\"odel numbering such that $0$ is not a G\"odel number of any object, and that G\"odel numbers of terms and formulas in which $\overline{n}$ occurs are larger than $n$. 
For each $\mathcal{L}_A$-formula $\varphi$, let $\gdl{\varphi}$ be the numeral for the G\"odel number of $\varphi$. 
Let $\{\varphi_k\}_{k \in \omega}$ be the repetition-free effective sequence of all formulas arranged in ascending order of G\"odel numbers. 
We assume that if $\varphi_k$ is a subformula of $\varphi_l$, then $k \leq l$.

We say a formula ${\sf Pr}_T(x)$ is a {\it provability predicate of $T$} if for any $n \in \omega$, $\PA \vdash {\sf Pr}_T(\overline{n})$ if and only if $n$ is the G\"odel number of some $T$-provable formula. 
We fix a primitive recursive formula $\Proof(x, y)$ which is a natural formalization of the relation ``$y$ is a $T$-proof of a formula $x$'' with the usual adequate properties. 
We may assume $\PA \vdash \forall x \forall y(\Proof(x, y) \to x \leq y)$. 
Let $\Prov(x)$ be the formula $\exists y \Proof(x, y)$. 
Then $\Prov(x)$ is a provability predicate of $T$ satisfying several familiar conditions such as $\PA \vdash \Prov(\gdl{\varphi \to \psi}) \to (\Prov(\gdl{\varphi}) \to \Prov(\gdl{\psi}))$ and $\PA \vdash \Prov(\gdl{\varphi}) \to \Prov(\gdl{\Prov(\gdl{\varphi})})$. 
Let $\Con_T$ be the sentence $\neg \Prov(\gdl{0=1})$ expressing the consistency of $T$. 
We say a formula $\Prf_T(x, y)$ is a {\it proof predicate of $T$} if $\Prf_T(x, y)$ satisfies the following conditions: 
\begin{enumerate}
	\item $\Prf_T(x, y)$ is primitive recursive, 
	\item $\PA \vdash \forall x (\Prov(x) \leftrightarrow \exists y \Prf_T(x, y))$, 
	\item for any $n \in \omega$ and formula $\varphi$, $\mathbb{N} \models \Proof(\gdl{\varphi}, \overline{n}) \leftrightarrow \Prf_T(\gdl{\varphi}, \overline{n})$, 
	\item $\PA \vdash \forall x \forall x' \forall y(\Prf_T(x, y) \land \Prf_T(x', y) \to x = x')$. 
\end{enumerate}
Here $\N$ is the standard model of arithmetic. 
The last clause means that our proof predicates are single conclusion ones. 
Our formula $\Proof(x, y)$ is one of proof predicates of $T$. 
For each proof predicate $\Prf_T(x, y)$ of $T$, the $\Sigma_1$ formula
\[
	\exists y (\Prf_T(x, y) \land \forall z \leq y \neg \Prf_T(\neg(x), z))
\]
is said to be the {\it Rosser provability predicate of $\Prf_T(x, y)$} or a Rosser provability predicate of $T$, where $\neg(x)$ is a term corresponding to a primitive recursive function calculating the G\"odel number of $\neg \varphi$ from the G\"odel number of a formula $\varphi$. 
Each Rosser provability predicate of $T$ is a provability predicate of $T$. 
Also the following proposition holds. 

\begin{prop}\label{RPP}
Let $\PR(x)$ be a Rosser provability predicate of $T$ and $\varphi$ be any formula. 
If $T \vdash \neg \varphi$, then $T \vdash \neg \PR(\gdl{\varphi})$. 
\end{prop}
As a consequence of Proposition \ref{RPP}, we have $T \vdash \neg \PR(\gdl{0=1})$. 

Let ${\sf Pr}_T(x)$ be any provability predicate of $T$. 
A mapping $f$ from the set of all modal formulas to the set of all $\mathcal{L}_A$-sentences is said to be an {\it arithmetical interpretation based on ${\sf Pr}_T(x)$} if $f$ satisfies the following conditions: 
\begin{enumerate}
	\item $f(\bot)$ is $0 = 1$,
	\item $f$ commutes with each propositional connective,
	\item $f(\Box A)$ is ${\sf Pr}_T(\gdl{f(A)})$. 
\end{enumerate}

The set $\PL({\sf Pr}_T) = \{A$ : $A$ is a modal formula and for all arithmetical interpretations $f$ based on ${\sf Pr}_T(x)$, 
$T \vdash f(A)\}$ is called the {\it provability logic of ${\sf Pr}_T(x)$}. 
One of the major achievements of the investigation of provability logics is Solovay's arithmetical completeness theorem (see \cite{Boo93,Smo85,Sol76}). 

\begin{thm}[Solovay's arithmetical completeness theorem for $\GL$]
If $T$ is $\Sigma_1$-sound, then $\PL(\Prov) = \GL$. 
\end{thm}

Provability logics of nonstandard provability predicates have been also studied by many authors. 
Feferman \cite{Fef60} found a nonstandard $\Sigma_2$ provability predicate ${\sf Pr}_T^F(x)$ such that $\KD \subseteq \PL({\sf Pr}_T^F)$ (see also \cite{Mon78,Vis89}). 
Shavrukov \cite{Sha94} found a Feferman-like $\Sigma_2$ provability predicate whose provability logic is exactly $\KD + \Box p \to \Box((\Box q \to q) \lor \Box p)$. 
Also it was proved in \cite{Kur18,Kur} that for each $L \in \{{\sf K}\} \cup \{{\sf K} + \Box (\Box^n p \to p) \to \Box p : n \geq 2\}$, there exists a $\Sigma_2$ provability predicate ${\sf Pr}_T(x)$ of $T$ such that $\PL({\sf Pr}_T)$ is precisely $L$ (The modal logics ${\sf K} + \Box (\Box^n p \to p) \to \Box p$ for $n \geq 2$ were introduced by Sacchetti \cite{Sac01}).

In this paper, we are interested in the provability logics $\PL(\PR)$ of Rosser provability predicates $\PR(x)$. 
In particular, we study the situation where $\PL(\PR)$ is a normal modal logic. 
We introduce the following terminology. 

\begin{defn}
A Rosser provability predicate $\PR(x)$ of $T$ is {\it normal} if $\PL(\PR)$ is a normal modal logic. 
\end{defn}

It is easy to show the following proposition. 

\begin{prop}\label{KDP}
For any Rosser provability predicate $\PR(x)$ of $T$, the following are equivalent: 
\begin{enumerate}
	\item $\PR(x)$ is normal. 
	\item $\KD \subseteq \PL(\PR)$. 
	\item $T \vdash \PR(\gdl{\varphi \to \psi}) \to (\PR(\gdl{\varphi}) \to \PR(\gdl{\psi}))$ for all sentences $\varphi$ and $\psi$. 
\end{enumerate}
\end{prop}

We can define a Rosser provability predicate which is not normal by using modal logical results of Guaspari and Solovay \cite{GS79} or Shavrukov \cite{Sha91}. 
On the other hand, Arai \cite{Ara90} defined a normal Rosser provability predicate (This was also mentioned by Shavrukov \cite{Sha91}). 
Thus whether $\PR(x)$ is normal or not depends on the choice of $\PR(x)$. 
Model theoretic properties of normal Rosser provability predicates were investigated by Kikuchi and Kurahashi \cite{KK17}.

We say an $\mathcal{L}_A$-formula $\varphi$ is {\it propositionally atomic} if it is not a Boolean combination of proper subformulas of $\varphi$. 
We prepare a new propositional variable $p_\varphi$ for each propositionally atomic formula $\varphi$. 
Then there exists a primitive recursive injection $I$ from $\mathcal{L}_A$-formulas to propositional formulas satisfying the following conditions: 
\begin{enumerate}
	\item $I(\varphi) \equiv p_\varphi$ for each propositionally atomic $\varphi$, 
	\item $I$ commutes with every propositional connective. 
\end{enumerate}
Let $X$ be any finite set of $\mathcal{L}_A$-formulas. 
We say $X$ is {\em propositionally satisfiable} if the set $I(X) = \{I(\varphi) : \varphi \in X\}$ of propositional formulas is satisfiable. 
An $\mathcal{L}_A$-formula $\psi$ is said to be a {\em tautological consequence (t.c.) of $X$} if $I(\psi)$ is a tautological consequence of $I(X)$. 
The above definitions are formalized in $\PA$, and $\PA$ can prove several familiar facts about them. 
For instance, $\PA$ proves that ``If $X \cup \{\varphi\}$ is not propositionally satisfiable for a finite set $X$ of formulas and a formula $\varphi$, then $\neg \varphi$ is a t.c.~of $X$.''
Define $P_{T, n}$ to be the finite set $\{\varphi : \N \models \exists y \leq \overline{n}\, {\sf Proof}_T(\gdl{\varphi}, y)\}$ of formulas. 
Then $\PA$ proves ``If a formula $\varphi$ is a t.c.~of $P_{T, n}$ for some $n$, then $\varphi$ is provable in $T$'', and so on.

\section{The arithmetical completeness theorem for $\KD$}\label{Sec:ACT}

In this section, we prove that there exists a normal Rosser provability predicate of $T$ whose provability logic is exactly the modal logic $\KD$. 

\begin{thm}\label{ACT1}
There exists a Rosser provability predicate $\PRg(x)$ of $T$ such that the following conditions hold:
\begin{enumerate}
	\item \textup{(Arithmetical soundness)} For any modal formula $A$, if $\KD \vdash A$, then $T \vdash f(A)$ for any arithmetical interpretation $f$ based on $\PRg(x)$. 
	In particular, $\PRg(x)$ is normal. 
	\item \textup{(Uniform arithmetical completeness)} There exists an arithmetical interpretation $f$ based on $\PRg(x)$ such that for any modal formula $A$, $\KD \vdash A$ if and only if $T \vdash f(A)$. 
\end{enumerate}
\end{thm}

Let $W = \omega \setminus \{0\}$. 
First, we define a primitive recursive function $h(x)$ by using the recursion theorem as follows: 
\begin{itemize}
	\item $h(0) = 0$, 
	\item $h(m+1) = \begin{cases} i & \text{if}\ h(m) = 0\ \&\ i \neq 0 \\
%	& \ \&\ \{j \in W : \neg S(\overline{j})\ \text{is a t.c.~of}\ P_{T, m}\} \neq \emptyset \\
	& \ \&\ i = \min \{j \in W : \neg S(\overline{j})\ \text{is a t.c.~of}\ P_{T, m}\},\\
			h(m) & \text{if $h(m) \neq 0$ or no $i$ as above exists.}
	\end{cases}$
\end{itemize}
Here $S(x)$ is the $\Sigma_1$ formula $\exists v(h(v) = x)$. 
The sentence $S(\overline{j})$ is propositionally atiomic because it is an existential sentence. 
Suppose that $P_{T, m}$ is propositionally satisfiable and $\neg S(\overline{j})$ is a t.c.~of $P_{T, m}$. 
Then there exists a formula $\varphi \in P_{T, m}$ containing $S(\overline{j})$ as a subformula. 
Hence $j \leq m$. 
Thus $\{j \in W : \neg S(\overline{j})\ \text{is a t.c.~of}\ P_{T, m}\} \neq \emptyset$ if and only if there exists $j \leq m$ such that $\neg S(\overline{j})$ is a t.c.~of $P_{T, m}$. 
Notice that this equivalence also holds when $P_{T, m}$ is not propositionally satisfiable. 
From this observation, for each $m$, whether $\{j \in W : \neg S(\overline{j})\ \text{is a t.c.~of}\ P_{T, m}\}$ is empty or not can be primitive recursively determined. 
This guarantees that $h$ is a primitive recursive function. 

By the definition of the function $h$, we obtain the following lemma. 

\begin{lem}\label{ACL1}\leavevmode
\begin{enumerate}
	\item $\PA \vdash \forall v \forall x (h(v) = x \land x \neq 0 \to \forall u \geq v\, h(u) = x)$. 
	\item $\PA \vdash \forall x \forall y(0 < x < y \land S(x) \to \neg S(y))$. 
	\item The sentences $\neg \Con_T$, $\exists x(\Prov(\gdl{\neg S(\dot{x})}) \land x \neq 0)$ and $\exists x (S(x) \land x \neq 0)$ are equivalent in $\PA$.
	\item For any $i \neq 0$, $T \nvdash \neg S(\overline{i})$. 
\end{enumerate}
\end{lem}

\begin{proof}
1 is proved by induction in $\PA$. 
2 follows from 1 immediately. 

3. $\PA \vdash \neg \Con_T \to \exists x(\Prov(\gdl{\neg S(\dot{x})}) \land x \neq 0)$ is obvious. 

We prove $\PA \vdash \exists x(\Prov(\gdl{\neg S(\dot{x})}) \land x \neq 0) \to \exists x(S(x) \land x \neq 0)$. 
We work in $\PA$: 
Suppose $\neg S(\overline{i})$ is provable in $T$ for $i \neq 0$. 
Then $\neg S(\overline{i})$ is contained in $P_{T, m}$ for some $m$. 
In this case, $\neg S(\overline{i})$ is a t.c.~of $P_{T, m}$. 
Hence $\{j \in W : \neg S(\overline{j})\ \text{is a t.c.~of}\ P_{T, n}\}$ is not empty for some $n$. 
For the least such $n$, $h(n + 1) \neq 0$. 

We prove $\PA \vdash \exists x(S(x) \land x \neq 0) \to \neg \Con_T$. 
We reason in $\PA$: 
Suppose $S(i)$ holds and $i \neq 0$. 
Then there exists $m$ such that $h(m) = 0$ and $h(m + 1) = i$. 
In this case, $i = \min \{j \in W : \neg S(\overline{j})\ \text{is a t.c.~of}\ P_{T, m}\}$, and hence $\neg S(\overline{i})$ is a t.c.~of $P_{T, m}$. 
Then $\neg S(\overline{i})$ is provable in $T$. 
On the other hand, the sentence $S(\overline{i})$ is provable in $T$ because it is a true $\Sigma_1$ sentence. 
Therefore $T$ is inconsistent. 

4. Suppose $T \vdash \neg S(\overline{i})$ for some $i \neq 0$. 
Then $\PA \vdash \exists x(\Prov(\gdl{\neg S(\dot{x})}) \land x \neq 0)$. 
By 3, $\PA \vdash \neg \Con_T$. 
This contradicts the consistency of $T$. 
Thus $T \nvdash \neg S(\overline{i})$. 
\end{proof}

Let $\{A_k\}_{k \in \omega}$ be a primitive recursive enumeration of all modal formulas that are not provable in $\KD$. 
From each $A_k$, we can primitive recursively construct a finite serial Kripke model $M_k = (W_k, \prec_k, \Vdash_k)$ in which $A_k$ is not valid by Theorem \ref{KCKD}. 
We may assume that $W_k$ and $W_l$ for $k \neq l$ are pairwise disjoint sets of natural numbers and $\bigcup_{k \in \omega} W_k = \omega \setminus \{0\} = W$. 
We define an infinite Kripke model $M = (W, \prec, \Vdash)$ which can be primitive recursively represented in $\PA$ as follows: 
\begin{enumerate}
%	\item $W = \omega \setminus \{0\}$, 
	\item $x \prec y$ if and only if for some $k \in \omega$, $x, y \in W_k$ and $x \prec_k y$, 
	\item $x \Vdash p$ if and only if for some $k \in \omega$, $x \in W_k$ and $x \Vdash_k p$. 
\end{enumerate}

For each $i \in W$, the set $\{j \in W : i \prec j\}$ is finite and nonempty because the Kripke model $M$ is a disjoint union of finite serial Kripke models. 
We can use the sentence $\bigvee_{i \prec j} S(\overline{j})$ which contains at least one disjunct. 

We define a primitive recursive function $g(x)$ which enumerates all theorems of $T$. 
The definition of $g$ consists of Procedures 1 and 2. 
The definition starts with Procedure 1 and the values of $g(0), g(1), \ldots$ are defined by referring to the values of the function $h$ in stages. 
The first time $h(m+1) \neq 0$, the construction of $g$ is switched to Procedure 2. 
In Procedure 2, $g$ outputs all formulas in stages. 

We start defining the function $g$. 
In the definition of $g$, we identify each formula with its G\"odel number. 

\vspace{0.1in}
Procedure 1. \\
Stage $1.m$: 
\begin{itemize}
		\item If $h(m + 1) = 0$, then
\[
	g(m) = \begin{cases}  \varphi & \text{if}\ m\ \text{is a proof of}\ \varphi\ \text{in} \ T,\ \text{that is},\ {\sf Proof}_T(\gdl{\varphi}, m)\ \text{holds}, \\
			0 & m\ \text{is not a proof of any formula in}\ T.
	\end{cases}
\]
	Go to Stage $1.(m+1)$. 
	\item If $h(m+1) \neq 0$, then go to Procedure 2. 
\end{itemize}

Procedure 2. \\
Let $m$ be the smallest number such that $h(m+1) \neq 0$. 
Let $i = \min \{j \in W : \neg S(\overline{j})\ \text{is a t.c.~of}\ P_{T, m}\}$ and let $X$ be the finite set $P_{T, m - 1} \cup \left\{\bigvee_{i \prec j} S(\overline{j})\right\}$ of formulas. 
Then $i = h(m+1)$. 
Let $\{\varphi_k\}_{k \in \omega}$ be the sequence of all formulas introduced in Section \ref{Sec:Pre}. 

We define the values of $g(m), g(m+1), \ldots$ and the numbers $\{t_k\}_{k \in \omega}$ simultaneously in stages. 
Let $t_0 = 0$. 
\\
Stage $2.k$: We distinguish the following three cases {\bf C1}, {\bf C2} and {\bf C3}. 
\begin{description}
	\item [C1] If $\varphi_k$ is a t.c.~of $X$, then let $g(m + t_k) = \varphi_k$ and $t_{k+1} = t_k + 1$.
	\item [C2] If $\varphi_k$ is not a t.c.~of $X$ and $\neg \varphi_k$ is a t.c.~of $X$, then let $g(m + t_k) = \neg \varphi_k$, $g(m + t_k + 1) = \varphi_k$ and $t_{k+1} = t_k + 2$. 
	\item [C3] If neither $\varphi_k$ nor $\neg \varphi_k$ is a t.c.~of $X$, then for each $0 \leq s \leq m$, let $g(m + t_k + s) = \overbrace{\neg \cdots \neg}^{m - s}\varphi_k$ and $t_{k+1} = t_k + m + 1$. 
\end{description}
Go to stage $2.(k+1)$.

The definition of $g$ has just been finished. 
Let $\Prf_g(x, y)$ be the primitive recursive formula $x = g(y) \land {\sf Fml}(x)$, where ${\sf Fml}(x)$ is a natural primitive recursive representation of ``$x$ is an $\mathcal{L}_A$-formula''. 
Also let ${\sf Pr}_g(x)$ and $\PRg(x)$ be the formula $\exists y \Prf_g(x, y)$ and the Rosser provability predicate of $\Prf_g(x, y)$, respectively. 
Actually, our formula $\Prf_g(x, y)$ is a proof predicate of $T$. 

\begin{lem}\label{ACL2}\leavevmode
\begin{enumerate}
	\item $\PA \vdash \forall x(\Prov(x) \leftrightarrow {\sf Pr}_g(x))$. 
	\item For any $n \in \omega$ and formula $\varphi$, $\N \models \Proof(\gdl{\varphi}, \overline{n}) \leftrightarrow \Prf_g(\gdl{\varphi}, \overline{n})$. 
\end{enumerate}
\end{lem}
\begin{proof}
1. It is clear that $\neg \exists x(S(x) \land x \neq 0) \to \forall x(\Prov(x) \leftrightarrow {\sf Pr}_g(x))$ is proved in $\PA$ by the definition of $g$. 
Also $\PA \vdash \exists x(S(x) \land x \neq 0) \to \forall x({\sf Pr}_g(x) \leftrightarrow {\sf Fml}(x))$ because each formula $\varphi_k$ is output at Stage $2.k$ in Procedure 2. 
Since $\PA \vdash \neg \Con_T \to \forall x(\Prov(x) \leftrightarrow {\sf Fml}(x))$, we have $\PA \vdash \exists x(S(x) \land x \neq 0) \to \forall x(\Prov(x) \leftrightarrow {\sf Pr}_g(x))$ by Lemma \ref{ACL1}.3. 
Thus we obtain $\PA \vdash \forall x(\Prov(x) \leftrightarrow {\sf Pr}_g(x))$. 

2. By Lemma \ref{ACL1}.3, $\neg \exists x(S(x) \land x \neq 0)$ is true in $\N$. 
Then $\Proof(\gdl{\varphi}, \overline{n})$ and $\Prf_g(\gdl{\varphi}, \overline{n})$ are equivalent in $\N$ by the definition of $g$. 
\end{proof}

\begin{lem}\label{ACL2.5}
Let $\psi$ be a formula. 
Then the following statement is provable in $\PA$: \\
``Let $i$ and $m$ be such that $h(m) = 0$ and $h(m+1) = i$. 
Let $X = P_{T, m-1} \cup \left\{\bigvee_{i \prec j}S(\overline{j})\right\}$. 
\begin{enumerate}
	\item For each $j_0 \succ i$, $\neg S(\overline{j_0})$ is not a t.c.~of $X$. 
	\item $X$ is propositionally satisfiable. 
	\item If a formula $\varphi$ is a t.c.~of $X$, then $\PRg(\gdl{\varphi})$ holds. 
	\item If $\psi$ is not a t.c.~of $X$, then $\neg \PRg(\gdl{\psi})$ holds. ''
\end{enumerate}
\end{lem}
\begin{proof}
We reason in $\PA$. 
Let $i$, $m$ and $X$ be as in the lemma. 
Then $i = \min \{j \in W : \neg S(\overline{j})\ \text{is a t.c.~of}\ P_{T, m}\}$. 
%Let $X = P_{T, m-1} \cup \left\{\bigvee_{i \prec j}S(\overline{j})\right\}$. 

1. Suppose that $\neg S(\overline{j_0})$ is a t.c.~of $X$ for some $j_0 \succ i$. 
Then $\bigvee_{i \prec j}S(\overline{j}) \to \neg S(\overline{j_0})$ is a t.c.~of $P_{T, m-1}$. 
Since $S(\overline{j_0}) \to \bigvee_{i \prec j}S(\overline{j})$ is a tautology, $S(\overline{j_0}) \to \neg S(\overline{j_0})$ is also a t.c.~of $P_{T, m-1}$. 
This means $\neg S(\overline{j_0})$ is a t.c.~of $P_{T, m-1}$. 
Then $\{j \in W : \neg S(\overline{j})$ is a t.c.~of $P_{T, m-1}\}$ is not empty, and hence $h(m) \neq 0$. 
This is a contradiction. 
Therefore $\neg S(\overline{j_0})$ is not a t.c.~of $X$. 

2. Since $\prec$ is serial, there exists at least one $j_0$ such that $j_0 \succ i$. 
Then $X \cup \{S(\overline{j_0})\}$ is propositionally satisfiable by 1. 
It follows that $X$ is propositionally satisfiable. 

3. Suppose $\varphi$ is a t.c.~of $X$. 
For the enumeration $\{\varphi_k\}_{k \in \omega}$ used in the definition of $g$, let $\varphi = \varphi_k$. 
Then $g(m + t_k) = \varphi$ by {\bf C1}. 
We prove that $g$ does not output $\neg \varphi$ before Stage $2.k$. 
Since $X$ is propositionally satisfiable by 2, $\neg \varphi \notin X$. 
Hence $\neg \varphi \notin P_{T, m-1}$. 
Therefore $\neg \varphi \notin \{g(0), \ldots, g(m-1)\}$ because the construction of $g$ executes Procedure 1 before Stage $1.m$. 

Since for each $k' < k$, $\neg \varphi = \neg \varphi_k$ is neither $\varphi_{k'}$ nor $\neg \varphi_{k'}$, $\neg \varphi$ is not output via ${\bf C1}$ nor via ${\bf C2}$ before Stage $2.k$. 
Also let $\varphi_l$ be a formula obtained by deleting zero or more leading negation symbols $\neg$ from $\varphi$. 
Then $l \leq k$ and exactly one of $\varphi_l$ and $\neg \varphi_l$ is a t.c.~of $X$. 
Thus $\neg \varphi$ is not output via ${\bf C3}$ at Stage $2.l$. 

Therefore $\neg \varphi \notin \{g(m), \ldots, g(m + t_k)\}$, and thus $\PRg(\gdl{\varphi})$ holds. 

4. Suppose that $\psi$ is not a t.c.~of $X$. 
Then $\psi \notin \{g(0), \ldots, g(m-1)\}$ since $\psi$ is not contained in $X$. 
We distinguish the following two cases. 
\begin{itemize}
	\item Case 1: $\neg \psi$ is a t.c.~of $X$. 
	Let $\psi = \varphi_k$. 
	Then $g(m + t_k) = \neg \psi$ by {\bf C2}. 
	We prove that $\psi$ is not output before Stage $2.k$. 

	Since $\psi \neq \varphi_{k'}$ for all $k' < k$, $\psi$ is not output by {\bf C1} before Stage $2.k$. 
	If $\psi$ is not a negated formula, then $\psi$ is not output via {\bf C2} nor via {\bf C3} before Stage $2.k$. 

	We assume that $\psi$ is of the form $\neg \varphi_l$ for some $l < k$. 
	Then $\varphi_l$ is a t.c.~of $X$. 
	Hence $g(m + t_l) = \varphi_l$ by {\bf C1}, and $\psi$ is not output at Stage $2.l$. 
	For each $l' < l$, $\psi = \neg \varphi_l$ is neither $\varphi_{l'}$ nor $\neg \varphi_{l'}$, so $\psi$ is not output by {\bf C2} before Stage $2.l$. 
	Also let $\varphi_p$ be a formula obtained by deleting zero or more leading $\neg$'s from $\varphi_l$, then $p \leq l$ and exactly one of $\varphi_p$ and $\neg \varphi_p$ is a t.c.~of $X$. 
	Hence $g$ does not output $\psi$ by {\bf C3} at Stage $2.p$. 
	Thus $\psi$ is not output via {\bf C2} nor via {\bf C3} before Stage $2.k$.

	We conclude that $\psi \notin \{g(m), \ldots, g(m + t_k)\}$. 
	Therefore $\PRg(\gdl{\psi})$ does not hold. 
	
	\item Case 2: $\neg \psi$ is not a t.c.~of $X$. 
	Let $\varphi_k$ be the formula obtained by deleting all leading $\neg$'s from $\psi$. 
	Then $\psi$ cannot be obtained by adding $\neg$'s to $\varphi_p$ for all $p < k$. 			Hence $g$ does not output $\psi$ before Stage $2.k$ by the definition of $g$. 
	Since neither $\varphi_k$ nor $\neg \varphi_k$ is a t.c.~of $X$, $g(m + t_k + s) = \overbrace{\neg \cdots \neg}^{m - s}\varphi_k$ for $0 \leq s \leq m$ by {\bf C3}. 
	Let $n$ be the number of deleted negation symbols from $\psi$. 
	Notice that $\neg S(\overline{i})$ is not a t.c.~of $P_{T, n}$ by Lemma \ref{ACL1}.4 (because $n$ is standard). 
	Hence $m > n$ holds. 
	Thus for $s = m - n - 1$, $g(m + t_k + s) = \neg \psi$ and $g(m + t_k + s + 1) = \psi$. 
	Therefore $\PRg(\gdl{\psi})$ does not hold. 
\end{itemize}
\end{proof}

\begin{lem}\label{ACL3}Let $i, k \in W$ and suppose $i \prec k$. 
\begin{enumerate}
	\item $\PA \vdash S(\overline{i}) \to \PRg \left(\gdl{\bigvee_{i \prec j}S(\overline{j})}\right)$. 
	\item $\PA \vdash S(\overline{i}) \to \neg \PRg(\gdl{\neg S(\overline{k})})$. 
\end{enumerate}
\end{lem}
\begin{proof}
Suppose $i \prec k$. 
We reason in $\PA + S(\overline{i})$: 
Let $m$ be such that $h(m) = 0$ and $h(m+1) = i$. 
Let $X = P_{T, m-1} \cup \left\{\bigvee_{i \prec j}S(\overline{j})\right\}$. 

Since $\bigvee_{i \prec j}S(\overline{j})$ is a t.c.~of $X$, $\PRg\left(\gdl{\bigvee_{i \prec j}S(\overline{j})} \right)$ holds by Lemma \ref{ACL2.5}.3. 
Also $\neg S(\overline{k})$ is not a t.c.~of $X$ by Lemma \ref{ACL2.5}.1. 
Therefore $\neg \PRg(\gdl{\neg S(\overline{k})})$ holds by Lemma \ref{ACL2.5}.4 (because $\neg S(\overline{k})$ is standard). 
\end{proof}

\begin{lem}\label{ACL4}
For any $\varphi$ and $\psi$, $\PA \vdash \PRg(\gdl{\varphi \to \psi}) \to (\PRg(\gdl{\varphi}) \to \PRg(\gdl{\psi}))$. 
\end{lem}
\begin{proof}
Since $\PA + \Con_T \vdash \Prov(\gdl{\varphi}) \to \neg \Prov(\gdl{\neg \varphi})$, we have $\PA + \Con_T \vdash {\sf Pr}_g(\gdl{\varphi}) \to \neg {\sf Pr}_g(\gdl{\neg \varphi})$ by Lemma \ref{ACL2}. 
Thus $\PA + \Con_T \vdash {\sf Pr}_g(\gdl{\varphi}) \leftrightarrow \PRg(\gdl{\varphi})$. 
It follows that $\PA + \Con_T \vdash \PRg(\gdl{\varphi \to \psi}) \to (\PRg(\gdl{\varphi}) \to \PRg(\gdl{\psi}))$. 

Then it suffices to prove $\PA + \neg \Con_T \vdash \PRg(\gdl{\varphi \to \psi}) \to (\PRg(\gdl{\varphi}) \to \PRg(\gdl{\psi}))$. 
We work in $\PA + \neg \Con_T$: 
By Lemma \ref{ACL1}.3, there exists $i \neq 0$ such that $S(i)$ holds. 
Then there exists $m$ such that $h(m) = 0$ and $h(m + 1) = i$. 
Let $X = P_{T, m - 1} \cup \left\{\bigvee_{i \prec j} S(\overline{j}) \right\}$. 

Suppose $\PRg(\gdl{\varphi \to \psi})$ and $\PRg(\gdl{\varphi})$. 
Then $\varphi \to \psi$ and $\varphi$ are t.c.'s of $X$ by Lemma \ref{ACL2.5}.4 (because $\varphi \to \psi$ and $\varphi$ are standard). 
Hence $\psi$ is also a t.c.~of $X$. 
We conclude $\PRg(\gdl{\psi})$ by Lemma \ref{ACL2.5}.3. 
\end{proof}

Define $f$ to be the arithmetical interpretation based on $\PRg(x)$ by $f(p) \equiv \exists x (S(x) \land x \neq 0 \land x \Vdash \gdl{p})$ for each propositional variable $p$. 

\begin{lem}\label{ACL6}Let $i \in W$ and $A$ be any modal formula. 
\begin{enumerate}	
	\item If $i \Vdash A$, then $\PA \vdash S(\overline{i}) \to f(A)$. 
	\item If $i \nVdash A$, then $\PA \vdash S(\overline{i}) \to \neg f(A)$. 
\end{enumerate}
\end{lem}
\begin{proof}
We prove 1 and 2 simultaneously by induction on the construction of $A$. 
We give a proof only for the case that $A$ is of the form $\Box B$. 

1. Suppose $i \Vdash \Box B$. 
Then for any $j \succ i$, $j \Vdash B$. 
By induction hypothesis, $\PA \vdash \bigvee_{i \prec j} S(\overline{j}) \to f(B)$. 
By Lemmas \ref{ACL2} and \ref{ACL4}, $\PA \vdash \PRg\left(\gdl{\bigvee_{i \prec j} S(\overline{j})} \right) \to \PRg(\gdl{f(B)})$. 
Since $\PA \vdash S(\overline{i}) \to \PRg\left(\gdl{\bigvee_{i \prec j} S(\overline{j})}\right)$ by Lemma \ref{ACL3}, we obtain $\PA \vdash S(\overline{i}) \to f(\Box B)$. 

2. Suppose $i \nVdash \Box B$. 
Then there exists $j \in W$ such that $i \prec j$ and $j \nVdash B$. 
By induction hypothesis, we have $\PA \vdash S(\overline{j}) \to \neg f(B)$. 
By Lemmas \ref{ACL2} and \ref{ACL4}, $\PA \vdash \neg \PRg(\gdl{\neg S(\overline{j})}) \to \neg \PRg(\gdl{f(B)})$. 
Since $i \prec j$, we obtain $\PA \vdash S(\overline{i}) \to \neg \PRg(\gdl{\neg S(\overline{j})})$ by Lemma \ref{ACL3}. 
Therefore $\PA \vdash S(\overline{i}) \to \neg f(\Box B)$. 
\end{proof}

We finish the proof of Theorem \ref{ACT1}. 

\begin{proof}[Proof of Theorem \ref{ACT1}]\leavevmode

1. Proposition \ref{KDP} indicates that arithmetical soundness follows from Lemmas \ref{ACL2} and \ref{ACL4}. 

2. Suppose $\KD \nvdash A$. 
Then there exists $i \in W$ such that $i \nVdash A$. 
By Lemma \ref{ACL6}, $\PA \vdash S(\overline{i}) \to \neg f(A)$. 
Since $T \nvdash \neg S(\overline{i})$ by Lemma \ref{ACL1}.4, we obtain $T \nvdash f(A)$. 
\end{proof}

Notice that our arithmetical interpretation $f$ maps each propositional variable to a $\Sigma_1$ sentence. 
We say such an arithmetical interpretation a $\Sigma_1$ arithmetical interpretation. 
Then we obtain the following corollary. 

\begin{cor}
Let $A$ be any modal formula. 
The following are equivalent: 
\begin{enumerate}
	\item $\KD \vdash A$. 
	\item For any normal Rosser provability predicate $\PR(x)$ of $T$ and any arithmetical interpretation $f$ based on $\PR(x)$, $T \vdash f(A)$. 
	\item For any normal Rosser provability predicate $\PR(x)$ of $T$ and any $\Sigma_1$ arithmetical interpretation $f$ based on $\PR(x)$, $T \vdash f(A)$. 
\end{enumerate}
\end{cor}

\section{Normal modal logic $\KDR$}\label{Sec:KDR}

It is known that the formalized version of Proposition \ref{RPP} is provable in $\PA$, that is, for any formula $\varphi$, $\PA \vdash \Prov(\gdl{\neg \varphi}) \to \Prov(\gdl{\neg \PR(\gdl{\varphi})})$ (see \cite{Sha91}).
Relating to this observation, in this section, we consider the following condition for Rosser provability predicates: 
\begin{eqnarray}\label{D4}
	\text{For any sentence}\ \varphi,\ T \vdash \PR(\gdl{\neg \varphi}) \to \PR(\gdl{\neg \PR(\gdl{\varphi})}). 
\end{eqnarray}

We introduce a new normal modal logic $\KDR$. 
\begin{defn}
$\KDR = \KD + \Box \neg p \to \Box \neg \Box p$. 
\end{defn}

Since $\KD \nvdash \Box \neg p \to \Box \neg \Box p$, $\KDR$ is a proper extension of $\KD$, and hence condition (\ref{D4}) is not valid for some Rosser provability predicate by Theorem \ref{ACT1}. 

It is easy to show that the validity of the modal formula $\Box \neg p \to \Box \neg \Box p$ in a Kripke frame $\mathcal{F} = (W, \prec)$ is characterized by the condition 
\begin{eqnarray}\label{R}
	\forall x \forall y \in W(x \prec y \Rightarrow \exists z \in W(x \prec z\ \&\ y \prec z)).
\end{eqnarray}
We say a Kripke model $M = (W, \prec, \Vdash)$ is a {\it $\KDR$-model} if $\prec$ is serial and satisfies condition (\ref{R}). 
Then it is proved that $\KDR$ is sound and complete with respect to the class of all $\KDR$-models (This follows from Theorem 3 in Boolos \cite{Boo93} p.~89 because $\KDR$ is $\K\{(0, 0, 1, 1), (0, 1, 1, 1)\}$ in the terminology of Boolos). 

\begin{prop}\label{KCKDR}
For any modal formula $A$, the following are equivalent: 
\begin{enumerate}
	\item $\KDR \vdash A$. 
	\item $A$ is valid in all $\KDR$-models.
\end{enumerate}
\end{prop}

Here we give an alternative axiomatization of $\KDR$. 

\begin{defn}\leavevmode
\begin{enumerate}
	\item Let $\KR = \K + \Box \neg p \to \Box \neg \Box p$. 
	\item Let $\KR^+$ be the logic obtained by adding the inference rule $\dfrac{\Box A}{A}$ to $\KR$. 
\end{enumerate}
\end{defn}

\begin{prop}
For any modal formula $A$, the following are equivalent: 
\begin{enumerate}
	\item $\KDR \vdash A$. 
	\item $\KR^+ \vdash A$. 
\end{enumerate}
\end{prop}
\begin{proof}
$(1 \Rightarrow 2)$: It suffices to show $\KR^+ \vdash \neg \Box \bot$. 
Since $\KR \vdash \Box \neg \bot \to \Box \neg \Box \bot$ and $\K \vdash \Box \neg \bot$, we have $\KR \vdash \Box \neg \Box \bot$. 
Then $\KR^+ \vdash \neg \Box \bot$. 

$(2 \Rightarrow 1)$: It suffices to show that the rule $\dfrac{\Box A}{A}$ is admissible in $\KDR$. 
Suppose $\KDR \nvdash A$. 
Then there exists a $\KDR$-model $M = (W, \prec, \Vdash)$ in which $A$ is not valid by Proposition \ref{KCKDR}. 
Let $M' = (W', \prec', \Vdash')$ be the Kripke model defined as follows: 
\begin{itemize}
	\item $W' = W \cup \{0\}$, where $0$ is an element not contained in $W$, 
	\item $\prec' = \prec \cup \{(0, w) : w \in W\}$, 
	\item $0 \Vdash' p$; and $x \Vdash' p \iff x \Vdash p$ for $x \in W$. 
\end{itemize}
Then $M'$ is also a $\KDR$-model. 
Since $0 \nVdash \Box A$, we obtain $\KDR \nvdash \Box A$ by Proposition \ref{KCKDR} again. 
\end{proof}

In Theorem \ref{KDRT1} below, we prove the existence of a normal Rosser provability predicate $\PR(x)$ satisfying $\KDR \subseteq \PL(\PR)$. 
%Hence we obtain a Rosser provability predicate whose provability logic is a proper extention of $\KD$. 
From the following proposition, in a construction of a Rosser provability predicate $\PR(x)$ satisfying $\KDR \subseteq \PL(\PR)$, we have to avoid the situation $T \vdash \neg \Con_T \to (\PR(\gdl{\varphi}) \lor \PR(\gdl{\neg \varphi}))$. 

\begin{prop}\label{KDRP3}
If $\KDR \subseteq \PL(\PR)$, then for some formula $\varphi$, $T$ cannot prove $\neg \Con_T \to (\PR(\gdl{\varphi}) \lor \PR(\gdl{\neg \varphi}))$. 
\end{prop}
\begin{proof}
Suppose $\KDR \subseteq \PL(\PR)$. 
Let $\varphi$ be a sentence satisfying the equivalence $T \vdash \varphi \leftrightarrow \neg \PR(\gdl{\varphi})$. 
Since $\KDR \subseteq \PL(\PR)$, we have $T \vdash \PR(\gdl{\neg \varphi}) \to \PR(\gdl{\neg \PR(\gdl{\varphi})})$. 
Then $T \vdash \PR(\gdl{\neg \varphi}) \to \PR(\gdl{\varphi})$. 
We obtain $T \vdash \neg \PR(\gdl{\neg \varphi})$ because $T \vdash \PR(\gdl{\neg \varphi}) \to \neg \PR(\gdl{\varphi})$. 
Assume towards a contradiction that $T$ proves $\Con_T \to (\PR(\gdl{\varphi}) \lor \PR(\gdl{\neg \varphi}))$.
Then $T \vdash \neg \Con_T \to \PR(\gdl{\varphi})$, and hence $T \vdash \neg \Con_T \to \neg \varphi$. 
By the formalized version of Rosser's first incompleteness theorem, we have $T \vdash \Prov(\gdl{\neg \varphi}) \to \neg \Con_T$ (see \cite{Lin03}). 
Thus $T \vdash \Prov(\gdl{\neg \varphi}) \to \neg \varphi$. 
By L\"ob's theorem, we obtain $T \vdash \neg \varphi$. 
Then $T$ is inconsistent by Rosser's first incompleteness theorem. 
This is a contradiction. 
Therefore $T \nvdash \Con_T \to (\PR(\gdl{\varphi}) \lor \PR(\gdl{\neg \varphi}))$.\end{proof}

We prove the main theorem of this section. 

\begin{thm}\label{KDRT1}
There exists a normal Rosser provability predicate $\PRgp(x)$ of $T$ such that $\KDR \subseteq \PL(\PRgp)$. 
\end{thm}

\begin{proof}
We define a primitive recursive function $g'(x)$ in stages. 
The corresponding formulas $\Prf_{g'}(x, y)$, ${\sf Pr}_{g'}(x)$ and $\PRgp(x)$ are defined as in the proof of Theorem \ref{ACT1}. 
%Then we prove $\KDR \subseteq \PL(\PRgp)$. 
In the definition of $g'$, we can refer to our defining formula $\PRgp(x)$ with the aid of the recursion theorem. 

For each natural number $m$, let $F_m$ be the set of all formulas whose G\"odel numbers are less than or equal to $m$. 
First, we define an increasing in $i$ sequence $\{Y_m^i\}_{i \in \omega}$ of finite sets of formulas recursively as follows: 
\begin{itemize}
	\item 	$Y_m^0 = P_{T, m}$, 
	\item $Y_m^{i+1} = Y_m^i \cup \{\neg \PRgp(\gdl{\varphi}) : \neg \PRgp(\gdl{\varphi}) \in F_m\ \&\ \neg \varphi\ \text{is a t.c.~of}\ Y_m^i\}$. 
\end{itemize}
Also let $Y_m = \bigcup_{i \in \omega} Y_m^i$. 

Notice that the definition of $Y_m$ refers to the formula $\PRgp(x)$. 
It is easy to see that for any $i \in \omega$, $Y_m^i \subseteq F_m$. 
Then $Y_m$ is in fact the finite set $\bigcup_{i \leq m} Y_m^i$ because $F_m$ contains at most $m$ formulas. 

We start defining our function $g'$. 

\vspace{0.1in}
Procedure 1. \\
Stage $1.m$: 
\begin{itemize}
		\item If $Y_m$ is propositionally satisfiable, then
\[
	g'(m) = \begin{cases}  \varphi & \text{if}\ m\ \text{is a proof of}\ \varphi\ \text{in} \ T,\ \text{that is},\ {\sf Proof}_T(\gdl{\varphi}, m)\ \text{holds}, \\
			0 & m\ \text{is not a proof of any formula in}\ T.
	\end{cases}
\]
	Go to Stage $1.(m+1)$. 
	\item If $Y_m$ is not propositionally satisfiable, then go to Procedure 2. 
\end{itemize}

Procedure 2. \\
Let $m$ be the least number such that $Y_m$ is not propositionally satisfiable. 
Notice that $m > 0$ because $Y_0 = P_{T, 0} = \emptyset$ is propositionally satisfiable. 
Let $X = Y_{m-1}$. 
Then $X$ is propositionally satisfiable. 
We define the values of $g'(m + t)$ for $t \geq 0$ by copying the corresponding part of Procedure 2 in the definition of the function $g$ in our proof of Theorem \ref{ACT1} with the present definition of $X$. 
The definition of $g'$ is completed. 

\begin{lem}\label{KDRL0.5}
$\PA$ proves the following statement: \\
``Let $m$ be the least number such that $Y_m$ is not propositionally satisfiable and let $X = Y_{m-1}$. 
\begin{enumerate}
	\item For any formula $\varphi \in F_m$, $\varphi$ is a t.c.~of $X$ if and only if $\PRgp(\gdl{\varphi})$ holds. 
	\item Every element of $X$ is provable in $T$.
	\item Suppose for all $j < i$, there exists no formula $\varphi$ such that $\neg \PRgp(\gdl{\varphi})$ is in $F_m$, $\varphi$ is a t.c.~of $X$ and $\neg \varphi$ is a t.c.~of $Y_m^j$. 
Then every element of $Y_m^i$ is provable in $T$.''
\end{enumerate}
\end{lem}
\begin{proof}
We proceed in $\PA$. 
Let $m$ be the least number such that $Y_m$ is not propositionally satisfiable and let $X = Y_{m-1}$. 
Then the construction of $g'$ goes to Procedure 2 at Stage $1.m$. 

1. Let $\varphi$ be any formula with $\varphi \in F_m$. 

$(\rightarrow)$: Suppose $\varphi$ is a t.c.~of $X$. 
Then $\neg \varphi \notin P_{T, m-1}$ because $P_{T, m-1} \subseteq X$ and $X$ is propositionally satisfiable. 
By tracing our proof of Lemma \ref{ACL2.5}.3, we can show that $\PRgp(\gdl{\varphi})$ holds. 

$(\leftarrow)$: Suppose $\varphi$ is not a t.c.~of $X$. 
%Then $\varphi \notin P_{T, m-1}$. 
Then we can show that $\PRgp(\gdl{\varphi})$ does not hold  by tracing our proof of Lemma \ref{ACL2.5}.4 for $\psi = \varphi$. 
The only difference between the proofs are the following: In our proof of Lemma \ref{ACL2.5}.4, the number of leading $\neg$'s of $\psi$ is less than $m$ because $\psi$ is standard, and here the number of leading $\neg$'s of $\varphi$ is less than $m$ because $\varphi \in F_m$. 

2. We prove by induction on $j$ that for all $j$, every element of $Y_{m-1}^j$ is provable in $T$. 
For $j = 0$, the statement trivially holds because $Y_{m-1}^0 = P_{T, m-1}$. 
Suppose that every formula contained in $Y_{m-1}^j$ is $T$-provable. 
Let $\varphi$ be any formula such that $\neg \PRgp(\gdl{\varphi}) \in F_{m-1}$ and $\neg \varphi$ is a t.c.~of $Y_{m-1}^j$.  
Since $Y_{m-1}^j \subseteq X$ and $X$ is propositionally satisfiable, $\varphi$ is not a t.c.~of $X$. 
Since $\varphi \in F_m$, $\PRgp(\gdl{\varphi})$ does not hold by 1. 
Moreover, $g'$ outputs $\neg \varphi$ without having output $\varphi$ earlier, and this fact is provable in $T$. 
Thus $\neg \PRgp(\gdl{\varphi})$ is provable in $T$. 
Therefore every element of $Y_{m-1}^{j+1}$ is provable in $T$ because $Y_{m-1}^{j+1} = Y_{m-1}^j \cup \{\neg \PRgp(\gdl{\varphi}) : \neg \PRgp(\gdl{\varphi}) \in F_{m-1}\ \&\ \neg \varphi\ \text{is a t.c.~of}\ Y_{m-1}^j\}$. 

3. This is proved in a similar way as in our proof of 2. 
\end{proof}

\begin{lem}\label{KDRL1}\leavevmode
\begin{enumerate}
	\item $\PA \vdash \Con_T \leftrightarrow \forall x$``\/$Y_x$ is propositionally satisfiable''.
	\item For any $n \in \omega$, $\PA \vdash$``\/$Y_n$ is propositionally satisfiable''. 
\end{enumerate}
\end{lem}
\begin{proof}
1. We proceed in $\PA$. 

$(\rightarrow)$: Let $m > 0$ be the least number such that $Y_m$ is not propositionally satisfiable, and let $X = Y_{m-1}$. 
Also let $i$ be the least number such that $Y_m^i$ is not propositionally satisfiable. 
We would like to show that $T$ is inconsistent. 

We prove that all elements of $Y_m^i$ are provable in $T$. 
If $i = 0$, this is obvious because $Y_m^0 = P_{T, m}$. 
We assume $i > 0$. 
Suppose, towards a contradiction, that some element of $Y_m^i$ is not provable in $T$. 
By Lemma \ref{KDRL0.5}.3, there exists the least $j < i$ such that for some formula $\varphi$, $\neg \PRgp(\gdl{\varphi}) \in F_m$, $\varphi$ is a t.c.~of $X$ and $\neg \varphi$ is a t.c.~of $Y_m^j$. 
By Lemma \ref{KDRL0.5}.2, $\varphi$ is $T$-provable. 
By the choice of $j$, every element of $Y_m^j$ is provable in $T$ by Lemma \ref{KDRL0.5}.3. 
Hence $\neg \varphi$ is also $T$-provable. 
Therefore $T$ is inconsistent. 
This contradicts the supposition. 
Thus all elements of $Y_m^i$ are $T$-provable. 

Since $Y_m^i$ is not propositionally satisfiable, $\neg \bigwedge Y_m^i$ is a tautology. 
Then both $\bigwedge Y_m^i$ and $\neg \bigwedge Y_m^i$ are provable in $T$, and hence $T$ is inconsistent. 

$(\leftarrow)$: Suppose that $T$ is inconsistent. 
Then $P_{T, m}$ is not propositionally satisfiable for some $m$. 
Therefore $Y_m$ is not propositionally satisfiable for some $m$. 

2. By 1, for any $n \in \omega$, the $\Sigma_1$ sentence ``$Y_n$ is propositionally satisfiable'' is true, and hence provable in $\PA$. 
\end{proof}

Our formula $\Prf_{g'}(x, y)$ is a proof predicate of $T$. 

\begin{lem}\label{KDRL2}\leavevmode
\begin{enumerate}
	\item $\PA \vdash \forall x(\Prov(x) \leftrightarrow {\sf Pr}_{g'}(x))$. 
	\item For any $n \in \omega$ and formula $\varphi$, $\N \models \Proof(\gdl{\varphi}, \overline{n}) \leftrightarrow \Prf_{g'}(\gdl{\varphi}, \overline{n})$. 
\end{enumerate}
\end{lem}
\begin{proof}
From Lemma \ref{KDRL1}, this is proved in a similar way as in our proof of Lemma \ref{ACL2}.
\end{proof}

\begin{lem}\label{KDRL3}
Let $\varphi$ and $\psi$ be any formulas. 
\begin{enumerate}
	\item $\PA + \Con_T \vdash \PRgp(\gdl{\varphi \to \psi}) \to (\PRgp(\gdl{\varphi}) \to \PRgp(\gdl{\psi}))$. 
	\item $\PA + \Con_T \vdash \PRgp(\gdl{\neg \varphi}) \to \PRgp(\gdl{\neg \PRgp(\gdl{\varphi})})$. 
\end{enumerate}
\end{lem}
\begin{proof}
Let $U = \PA + \Con_T$. 
By Lemma \ref{KDRL1}.1, $U$ proves $\forall x$``$Y_x$ is propositionally satisfiable''. 
Then by the definition of $g'$, the formulas $\Prov(x)$, ${\sf Pr}_{g'}(x)$ and $\PRgp(x)$ are all equivalent in $U$. 

1. Since $\PA \vdash \Prov(\gdl{\varphi \to \psi}) \to (\Prov(\gdl{\varphi}) \to \Prov(\gdl{\psi}))$, we have $U \vdash \PRgp(\gdl{\varphi \to \psi}) \to (\PRgp(\gdl{\varphi}) \to \PRgp(\gdl{\psi}))$. 

2. We reason in $U$: 
Suppose $\PRgp(\gdl{\neg \varphi})$ holds. 
Then $\neg \varphi$ is output by $g'$. 
Moreover, since $T$ is consistent, $g'$ outputs $\neg \varphi$ without having output $\varphi$ earlier, and also the latter fact is provable in $T$. 
Hence $T$ proves $\neg \PRgp(\gdl{\varphi})$. 
Therefore $\PRgp(\gdl{\neg \PRgp(\gdl{\varphi})})$ holds by the equivalence of $\Prov(x)$ and $\PRgp(x)$. 
\end{proof}

\begin{lem}\label{KDRL4}
Let $\varphi$ and $\psi$ be any formulas. 
\begin{enumerate}
	\item $\PA + \neg \Con_T \vdash \PRgp(\gdl{\varphi \to \psi}) \to (\PRgp(\gdl{\varphi}) \to \PRgp(\gdl{\psi}))$. 
	\item $\PA + \neg \Con_T \vdash \PRgp(\gdl{\neg \varphi}) \to \PRgp(\gdl{\neg \PRgp(\gdl{\varphi})})$. 
\end{enumerate}
\end{lem}
\begin{proof}
We reason in $\PA + \neg \Con_T$: 
By Lemma \ref{KDRL1}.1, there exists $m$ such that $Y_m$ is not propositionally satisfiable. 
Let $m$ be the least such number and let $X = Y_{m-1}$. 
By Lemma \ref{KDRL1}.2, $m$ is larger than the G\"odel numbers of $\varphi$, $\psi$, $\varphi \to \psi$ and $\neg \PRgp(\gdl{\varphi})$. 
That is, $\varphi$, $\psi$, $\varphi \to \psi$ and $\neg \PRgp(\gdl{\varphi})$ are in $F_m$. 

1. Suppose $\PRgp(\gdl{\varphi \to \psi})$ and $\PRgp(\gdl{\varphi})$ hold. 
Then $\varphi \to \psi$ and $\varphi$ are t.c.'s of $X$ by Lemma \ref{KDRL0.5}.1. 
Since $\psi$ is also a t.c.~of $X$, $\PRgp(\gdl{\psi})$ holds by Lemma \ref{KDRL0.5}.1 again. 

2. Suppose $\PRgp(\gdl{\neg \varphi})$ holds. 
Then $\neg \varphi$ is a t.c.~of $X$ by Lemma \ref{KDRL0.5}.1. 
Since $X = Y_{m-1} = \bigcup_{i \in \omega} Y_{m-1}^i$, $\neg \varphi$ is a t.c.~of $Y_{m-1}^i$ for some $i \in \omega$. 
Then $\neg \PRgp(\gdl{\varphi}) \in Y_{m-1}^{i+1} \subseteq X$ because $\neg \PRgp(\gdl{\varphi})$ is in $F_m$. 
Thus $\neg \PRgp(\gdl{\varphi})$ is also a t.c.~of $X$. 
By Lemma \ref{KDRL0.5}.1 again, $\PRgp(\gdl{\neg \PRgp(\gdl{\varphi})})$ holds. 
\end{proof}

By Lemma \ref{KDRL2}, $\PRgp(x)$ is a Rosser provability predicate of $T$. 
By Lemmas \ref{KDRL3} and \ref{KDRL4}, we conclude $\KDR \subseteq \PL(\PRgp)$. 
\end{proof}

Let ${\sf F}$ be Shavrukov's modal logic $\KD + \Box p \to \Box ((\Box q \to q) \lor \Box p)$ \cite{Sha94}. 
It is easy to see that ${\sf F}$ is included in $\Four \cap \KT$. 
Also we obtain the following proposition. 

\begin{prop}
$\KDR \subseteq \Five \cap {\sf F}$. 
Consequently, $\KDR \subseteq \Five \cap \Four \cap \KT$. 
\end{prop}
\begin{proof}
$\KDR \subseteq \Five$: This is because $\KD \vdash \Box \neg p \to \neg \Box p$ and $\Five \vdash \neg \Box p \to \Box \neg \Box p$. 

$\KDR \subseteq {\sf F}$: Since ${\sf F} \vdash \Box \neg p \to \Box \neg p \land \Box((\Box p \to p) \lor \Box \neg p)$, we have ${\sf F} \vdash \Box \neg p \to \Box(\neg \Box p \lor \Box \neg p)$. 
Since $\KD \vdash \Box(\neg \Box p \lor \Box \neg p) \to \Box \neg \Box p$, we conclude ${\sf F} \vdash \Box \neg p \to \Box \neg \Box p$. 
\end{proof}

In Theorem \ref{KDRT1}, we proved that there exists a Rosser provability predicate $\PR(x)$ such that $\KDR \subseteq \PL(\PR)$. 
On the other hand, $\PL(\PR)$ cannot include $\Five \cap \Four \cap \KT$. 

\begin{prop}
There exists no Rosser provability predicate $\PR(x)$ such that $\Five \cap \Four \cap \KT \subseteq \PL(\PR)$. 
\end{prop}
\begin{proof}
Let $\PR(x)$ be any normal Rosser provability predicate of $T$, and let $\varphi$ be a sentence satisfying $T \vdash \varphi \leftrightarrow \neg \PR(\gdl{\varphi})$. 
%Then $T \vdash \PR(\gdl{\varphi}) \leftrightarrow \PR(\gdl{\neg \PR(\gdl{\varphi})})$ and $T \vdash \PR(\gdl{\neg \varphi}) \leftrightarrow \PR(\gdl{\PR(\gdl{\varphi})})$ since $\KD \subseteq \PL(\PR)$. 
Let $\xi$ be the sentence $\neg \PR(\gdl{\varphi}) \to \PR(\gdl{\neg \PR(\gdl{\varphi})})$.
Then $T + \xi \vdash \varphi \to \PR(\gdl{\varphi})$ and $T + \xi \vdash \varphi \to \neg \varphi$. 
Hence $T + \xi \vdash \neg \varphi$. 
Then we have $T \vdash \PR(\gdl{\xi}) \to \PR(\gdl{\neg \varphi})$, and $T \vdash \PR(\gdl{\xi}) \to \neg \PR(\gdl{\varphi})$ because $\KD \subseteq \PL(\PR)$. 
Therefore $T + \PR(\gdl{\xi}) \vdash \varphi$. 
By combining this with $T + \xi \vdash \neg \varphi$, we obtain that $T + \xi \land \PR(\gdl{\xi})$ is inconsistent. 

Also let $\eta$ and $\gamma$ be the sentences $\PR(\gdl{\varphi}) \to \PR(\gdl{\PR(\gdl{\varphi})})$ and $\PR(\gdl{\varphi}) \to \varphi$, respectively. 
We can show that the theories $T + \eta$ and $T + \gamma$ prove $\varphi$. 
Then $T + \PR(\gdl{\eta})$ and $T + \PR(\gdl{\gamma})$ prove $\neg \varphi$. 
Therefore the theories $T + \eta \land \PR(\gdl{\eta})$ and $T + \gamma \land \PR(\gdl{\gamma})$ are also inconsistent. 
Let $\alpha$ be the sentence 
\[
	(\xi \land \PR(\gdl{\xi})) \lor (\eta \land \PR(\gdl{\eta})) \lor (\gamma  \land \PR(\gdl{\gamma})).
\] 
Then we have shown $T \vdash \neg \alpha$. 
On the other hand, the sentence $\alpha$ is an arithmetical instance of a modal formula which is in $\Five \cap \Four \cap \KT$. 
Therefore we conclude $\Five \cap \Four \cap \KT \nsubseteq \PL(\PR)$. 
\end{proof}

We propose a question concerning the logics $\KDR$ and ${\sf F}$. 

\begin{prob}\leavevmode
Are there respective normal Rosser provability predicates $\PR(x)$ of $T$ satisfying each of the following conditions?
\begin{enumerate}
	\item $\KD \subsetneq \PL(\PR) \subsetneq \KDR$.
	\item $\KD \subsetneq \PL(\PR)$ and $\PL(\PR)$ is incomparable with $\KDR$. 
	\item $\KDR = \PL(\PR)$. 
	\item ${\sf F} \subseteq \PL(\PR)$ or ${\sf F} = \PL(\PR)$. 
	\item $\KDR \subsetneq \PL(\PR) \subsetneq \Five \cap \Four \cap \KT$. 
	\item $\KDR \subsetneq \PL(\PR) \nsubseteq \Five \cap \Four \cap \KT$. 
	\item $\KDR \subsetneq \PL(\PR) \subsetneq {\sf F}$. 
	\item $\KDR \subsetneq \PL(\PR) \nsubseteq {\sf F}$. 
\end{enumerate}
\end{prob}

\begin{rem}
Since all consistent normal extensions of $\KD$ are contained in the trivial modal logic (see Lemma 3.3 of Hughes and Cresswell \cite{HC96}), for any two normal Rosser provability predicates ${\sf Pr}_0^R(x)$ and ${\sf Pr}_1^R(x)$ of $T$, $\PL({\sf Pr}_0^R)$ is consistent with $\PL({\sf Pr}_1^R)$ because they have a common consistent extension. 
\end{rem}

In Theorem \ref{ACT1}, we proved that there exists a Rosser provability predicate $\PR(x)$ such that $\PL(\PR) = \KD$, but our proof does not guarantee that $\PR(x)$ satisfies the following stronger version of the principle ({\bf K}): 
\begin{description}
	\item ({\bf K'}): $T \vdash \forall x \forall y(\PR(x \dot{\to} y) \to (\PR(x) \to \PR(y)))$. 
\end{description}
Here $x \dot{\to} y$ is a primitive recursive term corresponding to a primitive recursive function calculating the G\"odel number of $\varphi \to \psi$ from G\"odel numbers of $\varphi$ and $\psi$. 
A Rosser provability predicate created by Arai \cite{Ara90} satisfies the principle ({\bf K'}). We can make the same comment for Theorem \ref{KDRT1}, and so we close this paper with the following problem.

\begin{prob}\leavevmode
\begin{enumerate}
	\item Is there a Rosser provability predicate $\PR(x)$ such that $\PL(\PR) = \KD$ and $\PR(x)$ satisfies {\rm ({\bf K'})}?
	\item Is there a Rosser provability predicate $\PR(x)$ such that $\KDR \subseteq \PL(\PR)$ and $\PR(x)$ satisfies {\rm ({\bf K'})}?
\end{enumerate}
\end{prob}

\bibliographystyle{plain}
\bibliography{ref}

\end{document}